\documentclass{amsart}
\usepackage{amssymb, amsmath, latexsym}

\renewcommand{\baselinestretch}{\baselinestretch}
\renewcommand{\baselinestretch}{1.1}
\numberwithin{equation}{section}

\newcommand{\z}{\mathbb Z}
\newcommand{\norm}{\mathfrak N}

\newcommand{\x}{\mathbf x}

\newcommand{\be}{\mathbf e}
\newcommand{\bv}{\mathbf v}
\newcommand{\bw}{\mathbf w}
\newcommand{\bu}{\mathbf u}
\newcommand{\oc}{\overline{c}}
\newcommand{\od}{\overline{\alpha}}
\newcommand{\p}{\mathfrak P}
\newcommand{\gen}{\text{gen}}

\newtheorem{thm}{Theorem}[section]
\newtheorem{lem}[thm]{Lemma}
\newtheorem{cor}[thm]{Corollary}
\newtheorem{prop}[thm]{Proposition}

\theoremstyle{definition}
\newtheorem{defn}[thm]{Definition}

\theoremstyle{remark}

\numberwithin{equation}{section}

\begin{document}

\title{On a Waring's problem for integral quadratic and hermitian forms}

\author{Constantin N. Beli}
\address{Institute of Mathematics Simion Stoilow of the Romanian Academy, Calea Grivitei 21, RO-010702 Bucharest, Romania}
\email{Constantin.Beli@imar.ro}

\author{Wai Kiu Chan}
\address{Department of Mathematics and Computer Science, Wesleyan University, Middletown CT, 06459, USA}
%\curraddr{}
\email{wkchan@wesleyan.edu}

\author{Mar\'ia In\'es Icaza}
\address{Instituto de Matem\'atica y F\'isica, Universidad de Talca, Casilla 747, Talca, Chile.}
\email{icazap@inst-mat.utalca.cl}

% author two information
\author{Jingbo Liu}
\address{Department of Mathematics, University of Hong Kong, Pokfulam Road, Hong Kong}
\email{jliu02@hku.hk}
%\thanks{}

\subjclass[2010]{Primary 11E12, 11E25, 11E39}

\keywords{Waring's problem, Sums of squares, Sums of norms}

%\subjclass[2000]{Primary 11E12, 11E20} \keywords{quadratic
%equations}

\begin{abstract}
For each positive integer $n$, let $g_\z(n)$ be the smallest integer such that if an integral quadratic form in $n$
variables can be written as a sum of squares of integral linear forms, then it can be written as a sum of $g_\z(n)$
squares of integral linear forms.  We show that as $n$ goes to infinity, the growth of $g_\z(n)$ is at most an exponential of $\sqrt{n}$.
Our result improves the best known upper bound on $g_\z(n)$ which is in the order of an exponential of $n$.
We also define an analogous number $g_{\mathcal O}^*(n)$ for writing hermitian forms over the ring of integers $\mathcal O$ of
an imaginary quadratic field as sums of norms of integral linear forms, and when the class number of the imaginary quadratic field is 1,
we show that the growth of $g_{\mathcal O}^*(n)$ is at most an exponential of $\sqrt{n}$. We also improve results of Conway-Sloane \cite{cs} and
Kim-Oh \cite{Kim-Oh05} on $s$-integral lattices.
\end{abstract}

\maketitle

\section{Introduction} \label{introduction}

Representations of integers as sums of (integer) squares is a question which has piqued the interest of many mathematicians for centuries.
Because of the well-known work of Fermat, Euler, Legendre, and Lagrange on sums of squares, it is now known that every positive integer is
a sum of at most four squares.  This result has been generalized in many different ways.  For example,  the {\em Pythagoras number} of
a ring $R$ is the smallest positive integer $p = p(R)$ such that every sum of squares of elements of $R$ is already a sum of $p$ squares
of elements of $R$.  The results we just mentioned about sums of squares is tantamount to saying that the Pythagoras number of $\z$ is 4.

There is also a higher dimensional generalization in terms of representations of integral quadratic forms.
An integral quadratic form $g(\mathbf y)$ in variables $\mathbf y = (y_1, \ldots, y_m)$ is said to be represented by
another integral quadratic form $f(\mathbf x)$ in variables $\mathbf x = (x_1, \ldots, x_n)$ ($n \geq m$) if there exists
an $m \times n$ integral matrix $T$ such that $g(\mathbf y) = f(\mathbf y T)$.  For the sake of convenience, for any positive integer
$r$ we denote the quadratic form $x_1^2 + \cdots + x_r^2$ by its Gram
matrix $I_r$.  In general, a quadratic form is represented by $I_r$ if and only if it is the sum of $r$ squares of integral linear forms.  In this context, Lagrange's Four-Square Theorem simply says that every unary positive definite integral quadratic form is represented by $I_4$.  Mordell \cite{Mordell30} made the first step of generalization along this direction by proving that every positive
definite binary integral quadratic form is represented by $I_5$, and
that $5$ is the smallest number with this property.\footnote{In a
recent short  note \cite{Schinzel13}, A. Schinzel has reported and
filled in a gap in Mordell's proof.}
In the same paper, he posed the following what he called a {\em new Waring's Problem}: can every positive definite integral quadratic
form in $n$ variables be written as a sum of $n + 3$ squares of integral linear forms?  This was proven to be true when
$n \leq 5$ by Ko \cite{Ko37}, but around the same time Mordell \cite{Mordell37} showed that the 6-variable quadratic form
corresponding to the root system $E_6$ cannot be represented by {\em any} sum of squares.   Later Ko \cite{Ko39} showed
further that up to equivalence this 6-variable quadratic form is the only counterexample among all positive definite
6-variable integral quadratic forms.  This leads to the consideration of the set $\Sigma_\z(n)$ of all  integral
quadratic forms in $n$ variables that can be represented by some sums of squares and the following definition of
the ``$g$-invariants" of $\z$:
$$g_\z(n) := \min\{g : \mbox{ every quadratic form in } \Sigma_\z(n) \mbox{ is represented by } I_g \}.$$
All the results mentioned thus far can be summarized as $g_\z(n) = n + 3$ for $n \leq 5$.  Ko \cite{Ko39}
conjectured that $g_\z(6) = 9$, but this is disproved almost sixty years later by Kim-Oh \cite{Kim-Oh97}
who show that $g_\z(6)$ is actually equal to 10.  This is the last known exact value of $g_\z(n)$, although
explicit upper bounds for $g_\z(n)$ for $n \leq 20$ have been found; see \cite{Kim-Oh02} and \cite{Sasaki00}.

The definition of the $g$-invariants of $\z$ is a special case of the $g$-invariants of a commutative ring $A$ first
appeared in the third author's paper \cite{Icaza96}, although the special case where $A$ is a field is studied earlier
in \cite{BLOP86}.  Using a deep result in representations of positive definite integral quadratic forms \cite{HKK77},
the third author \cite{Icaza96} shows that when $\mathcal O$ is the ring of integers of a totally real number field,
$g_{\mathcal O}(n)$ is bounded above by a function of $n$ which is at least an exponential function of the class number
of $I_{n+3}$ as a quadratic form over $\mathcal O$.   The latter is expected to be growing extremely fast as $n$ increases.
Indeed, in the special case $\mathcal O = \z$, the class number of $I_{n+3}$ is at least in the order of $n^{n^2}$ by
considering the mass of the genus of $I_{n+3}$ \cite{MH73}.  A much better upper bound $g_\z(n) = O(3^{n/2}n \log n)$
is later obtained by Kim-Oh \cite{Kim-Oh05} which is the best upper bound on $g_\z(n)$ so far.

In this paper, we will improve the upper bound on $g_\z(n)$ by showing that its growth is at most an exponential of $\sqrt{n}$.
Before introducing the precise statement of this result, we note that our method applies to the analogous problem of representing
integral hermitian forms over an imaginary quadratic field by sums of norms.  Let $E$ be an imaginary quadratic field, $*$ be
its nontrivial Galois automorphism, and $\mathcal O$ be its ring of integers.  Let $I_r$ be the  hermitian form $x_1x_1^* + \cdots + x_rx_r^*$
over $\mathcal O$, the sum of $r$ norms.  For any positive integer $n$, let $\Sigma_{\mathcal O}^*(n)$ be the set of integral hermitian forms
over $\mathcal O$ in $n$ variables that are represented by sums of norms.  We define
$$g_{\mathcal O}^*(n) = \min\{g : \mbox{ every hermitian form in } \Sigma_{\mathcal O}^*(n) \mbox{ is represented by } I_g \}.$$
The finiteness of $g_{\mathcal O}^*(n)$ can be proved in the same way as $g_\z(n)$.  A proof of this is presented in the Appendix
when $E$ is replaced by an arbitrary CM field and the hermitian forms by the more general notion of hermitian $\mathcal O$-lattices.
To unify the discussion of the hermitian and quadratic cases together,  we allow the pair $(E, *)$ to be $(\mathbb Q, 1_{\mathbb Q})$
in the subsequent discussion.  In particular, a quadratic form over $\z$ is just a hermitian form over $\z$ with respect to the trivial
automorphism, and  $g_\z(n)$ is simply $g_\z^*(n)$ as a consequence.  Our main result can be stated as

\begin{thm}\label{main}
Let $E$ be $\mathbb Q$ or an imaginary quadratic field with class
number 1, and $\mathcal O$ be its ring  of integers.  Then for any
$\varepsilon >0$ we have
$$g_{\mathcal O}^*(n)=O(e^{(k_E+\varepsilon )\sqrt{n}}),$$
where $k_E$ is a constant depending on the field $E$.
\end{thm}

Explicitly, the constant $k_E$ appeared in Theorem \ref{main} is equal to $(4 + 4\sqrt{2})\sqrt{\beta_E}$, where $\beta_E$ is defined as
\begin{equation} \label{betaE}
\beta_E := \sup_{x\in E} \inf_{c \in \mathcal O} \vert \text{N}_{E/\mathbb Q}(x - c)\vert^{\frac{1}{[E:\mathbb Q]}},
\end{equation}
where $\text{N}_{E/\mathbb Q}$ is the norm from $E$ to $\mathbb Q$.
In literature, the number $\beta_E^{[E: \mathbb Q]}$ is called the Euclidean minimum of $E$.  It is clear from the definition that $\beta_{\mathbb Q} = \frac{1}{2}$.  The exact value of $\beta_E$ for a general imaginary quadratic field $E = \mathbb Q(\sqrt{-\ell})$, $\ell > 0$ squarefree, can be
easily deduced from a result of Dirichlet \cite{Dirichlet} (see also the survey \cite{Lem}):
$$\beta_E^2 = \begin{cases}
\frac{\ell + 1}{4} & \mbox{ if $\ell \equiv 1$ or $2$ mod 4};\\
\frac{(\ell + 1)^2}{16\ell} & \mbox{ if $\ell \equiv 3$ mod 4}.
\end{cases}$$

Let $\vert\,\,\vert$ be the valuation on $E$ that is the usual absolute value when restricted on $\mathbb Q$, and $\mathbb K$ be the completion of $E$ with respect to $\vert\,\,\vert$.    Let $\mathcal C_E$ be the set $\{z\in \mathbb K : \vert z \vert \leq \beta_E \}$.   As a straightforward consequence from the definition of $\beta_E$ and the denseness of $E$ in $\mathbb K$, we see that given any $z \in \mathbb K$, there exist $a \in \mathcal O$ and $\eta \in \mathcal C_E$ such that $z = a + \eta$.

We will show that every positive definite hermitian form over any one of those fields in the theorem is integrally equivalent to
a what we call ``balanced Hermite-Korkin-Zolotarev (HKZ) reduced" hermitian form, which will be crucial in obtaining our upper bounds
for $g_{\mathcal O}^*(n)$.  Our method does not apply to general imaginary quadratic fields.  However, an upper bound for
$g_{\mathcal O}^*(n)$ when $E$ is an arbitrary imaginary quadratic field, which is in the order of an exponential of $n$,
is obtained by the fourth author in her doctoral thesis \cite{Liu16}.

The rest of the paper is organized as follows.  In Section 2 we introduce the weakly reduced hermitian forms,
and establish upper bounds on the entries of the diagonal part of the Gram matrix of a weakly reduced hermitian form.
Section 3 contains some technical lemmas which will lead to the balanced HKZ reduced hermitian forms.  In Section \ref{neighbors} we will recall
some results on the theory of neighbors of hermitian forms due to Schiemann \cite{Sch}.    The derivation
of the final upper bound on $g_{\mathcal O}^*(n)$ will be presented in Section \ref{mainresults}.  In Section \ref{conway} we will apply our results to obtain a much improved lower bound for a function $\phi(s)$ defined using the $s$-integrable lattices, which were introduced by Conway and Sloane in \cite{cs}.  An Appendix is given at the end to provide a proof of the finiteness of $g_{\mathcal O}^*(n)$ when $E$ is a CM extension of a totally real number field.

\section{Weakly reduced hermitian forms}

We start with the more general assumption that $\mathcal O$ is a PID with field of fractions $E$ and $V$ is an $n$-dimensional vector space over $E$.  Let $\Lambda$ be an $\mathcal O$-lattice on $V$, that is, a finitely generated $\mathcal O$-module in $V$ such that $E\Lambda = V$.   Let $\bv_1$ be a primitive vector in $\Lambda$ and $V_2$ be a subspace of $V$ such that $V = E\bv_1 \oplus V_2$.  If $p_2: V \longrightarrow V_2$ is the projection of $V$ onto $V_2$ along the subspace $E\bv_1$, then $p_2(\Lambda)$ is a lattice on $V_2$.

\begin{lem} \label{martinet}
If $\bv_2, \ldots, \bv_n$ is a basis of $p_2(\Lambda)$ and for $2\leq i \leq n$, $\bu_i \in \Lambda$ is chosen such that $p_2(\bu_i) = \bv_i$, then $\bv_1, \bu_2 \ldots, \bu_n$ is a basis of $\Lambda$
\end{lem}
\begin{proof}
This is essentially \cite[Lemma 2.2.4]{Martinet} whose proof works in our slightly more general situation.
\end{proof}

Now, let us restrict our discussion to the case when $E$ is either $\mathbb Q$ or one of the nine imaginary quadratic fields with class number 1.  Let $*$ be the nontrivial automorphism of $E$ if $E$ is not $\mathbb Q$, and the identity map if $E$ is $\mathbb Q$.  From now on, a hermitian
form over $E$ is either a quadratic form when $E = \mathbb Q$ or a hermitian form with respect to $*$ if $E$ is not $\mathbb Q$.
If $f$ is a hermitian form in variables $x_1, \ldots, x_n$ over $E$, it can be written as $\sum_{1\leq i, j \leq n} a_{ij}x_ix_j^*$ where
$a_{ij} = a_{ji}^* \in E$.  The matrix $(a_{ij})$ is called the {\em Gram matrix} associated to $f$.  The discriminant of $f$,
denoted $d(f)$, is the determinant of the Gram matrix associated to $f$.  Let $\mu(f)$ be the (nonzero) minimum of $f$, i.e.,
$\mu(f) = \min\{f(\x): 0 \neq \x \in \mathcal O^n\}$.

If $a_1x_1 + \cdots + a_nx_n$ is a linear form over $E$,  we define
$$\norm(a_1x_1 + \cdots + a_nx_n) := (a_1x_1 + \cdots + a_nx_n)(a_1^*x_1^* + \cdots + a_n^*x_n^*),$$
which is a hermitian form over $E$.  For any matrix $A = (a_{ij})$ with entries in $E$, $A^* = (a_{ji}^*)$ denotes its ``conjugate transpose".

\begin{defn}
A positive definite hermitian form $f(x_1, \ldots, x_n)$ over $E$ is said to be weakly reduced if
$$f(x_1, \ldots, x_n)=\displaystyle\sum_{i=1}^nh_i \,\norm\left(x_i+\sum_{j=i+1}^nt_{ij}x_j\right),$$
with
\begin{equation}\label{hi}
h_i=\displaystyle\min_{(x_i,...,x_n)\in\mathcal{O}^{n-i+1}\setminus\{0\}}\displaystyle\sum_{j=i}^nh_j \, \norm\left(x_j+\sum_{k=j+1}^nt_{jk}x_k\right).
\end{equation}
\end{defn}

\begin{prop} \label{kz}
Every positive definite hermitian form over $E$ is integrally equivalent to a weakly reduced hermitian form.
\end{prop}
\begin{proof}
Our proof is similar to the strategy outlined in \cite[Page 60]{Martinet} for the case $E = \mathbb Q$.  Let $f(x_1, \ldots, x_n)$ be a hermitian
form over $E$ and $M$ be its Gram matrix.  Let $V$ be an $n$-dimensional vector space over $E$ with a basis $\{\be_1, \ldots, \be_n\}$.
We view $V$ as a hermitian space with the positive definite hermitian map $h: V\times V \longrightarrow E$ such that $h(\be_i, \be_j)$
is the $(i, j)$ entry of $M$. Let $\Lambda$ be the $\mathcal O$-lattice spanned by $\be_1, \ldots, \be_n$.  For any $\bv = x_1\be_1 + \cdots + x_n\be_n \in \Lambda$, $h(\bv, \bv)$ is equal to $f(x_1, \ldots, x_n)$.   We shall prove that $\Lambda$ has a basis with respect to which $h(\bv, \bv)$ is a weakly reduced hermitian form.

Let $0 \neq \bv_1 \in \Lambda$ be a minimal vector of $\Lambda$, that is, $h(\bv_1, \bv_1)$ is equal to $\mu(f)$, and $V_2$ be the orthogonal complement of $E\bv_1$ in $V$.  Using Lemma \ref{martinet} and an induction argument, we obtain an orthogonal basis $\bv_1, \ldots, \bv_n$ of $V$ and a basis $\bu_1, \ldots, \bu_n$ of $\Lambda$ such that $\bu_1 = \bv_1$,  $\bv_i$ is a minimal vector of $p_i(\Lambda)$ for $2 \leq i \leq n$, where $p_i$ is the orthogonal projection of $V$ onto the orthogonal complement of $\text{span}(\bv_1, \ldots, \bv_{i-1})$, and $p_i(\bv_i) = \bu_i$.  Moreover, for $2\leq i \leq n$,
$$\bu_i = \bv_i + \sum_{j = 1}^{i-1} t_{ji}\bv_j, \quad t_{ij} \in E.$$

Every vector $\bv$ in $\Lambda$ is of the form $\bv = \sum_{i=1}^n x_i \bu_i$, where $x_i \in \mathcal O$ for all $i$, which can be re-written as
\begin{equation} \label{formofv}
\bv = \sum_{i=1}^n \left(x_i + \sum_{j = i+1}^n t_{ij}x_j\right)\bv_i.
\end{equation}
Let $h_i = h(\bv_i, \bv_i)$.  Then, since $\bv_1, \ldots, \bv_n$ is an orthogonal basis, we have
\begin{equation*} %\label{newform}
h(\bv, \bv) = \sum_{i=1}^n h_i \,\norm\left(x_i + \sum_{j = i + 1}^n t_{ij}x_j\right).
\end{equation*}
It remains to show that each $h_i$ satisfies condition \eqref{hi} in the definition of weakly reduced hermitian forms.

Since a typical element of $\Lambda$ is of the form given by \eqref{formofv},  a typical element of $p_{i}(\Lambda )$ has the form $\bv'=\sum_{j=i}^n(x_j+\sum_{k=j+1}^nt_{jk}x_k)\bv_j$, with
$h(\bv',\bv')=\sum_{j=i}^nh_j\norm(x_j+\sum_{k=j+1}^nt_{jk}x_k)$. Therefore,
$$h_i = h(\bv_i, \bv_i) = \min_{\bv' \in p_{i}(\Lambda)\setminus\{0\}} h(\bv', \bv') = \displaystyle\min_{(x_i,...,x_n)\in\mathcal{O}^{n-i+1}\setminus\{0\}}\displaystyle\sum_{j=i}^nh_j \, \norm\left(x_j+\sum_{k=j+1}^nt_{jk}x_k\right).$$
This completes the proof of the proposition.
\end{proof}

The Gram matrix of a weakly reduced hermitian form is of the form $X^* H X$, where $H$ is a diagonal matrix whose entries are defined by
\eqref{hi} and $X$ is a upper triangular unipotent matrix.   In the case of $E = \mathbb Q$, Hermite-Korkin-Zolotarev reduction (\cite{KZ}, \cite{Martinet} and \cite{Z}) implies that $X$ can be chosen in such a way that its entries are bounded by $\frac{1}{2}$.  In the proof of Proposition \ref{kz}
we could have shown further that $X$ can be chosen so that its entries are bounded above by the Euclidean minimum $\beta_E$.
In our derivation of the upper bound on $g_{\mathcal O}^*(n)$, we will need {\em good} asymptotic bounds on the entries of both $X$ and $X^{-1}$.
However, the entries of $X^{-1}$ could be huge even if the entries of $X$ are bounded by $\beta_E$ as described above.  For instance,
if all entries of $X$ above the diagonal are $-\beta_E$, then the $(1,n)$ entry of $X^{-1}$ is $\beta_E(1 + \beta_E)^{n-2}$, which is in the
order of an exponential of $n$.  In the next section we will show that one can choose $X$ so that the entries of both $X$ and $X^{-1}$ are at
worst in the order of an exponential of $\sqrt{n}$, and that will be a crucial step in obtaining our main result.

For the rest of this section, we will concentrate on bounding the $h_i$'s defined by \eqref{hi}.
Let\footnote{When $E$ is an imaginary quadratic field, our $\gamma_{n,E}$ is the square root of the Hermite-Humbert constant
for $E$ defined by Icaza in \cite{Icaza97}.}
$$\gamma_{n, E} := \sup_f \frac{\mu(f)}{d(f)^{\frac{1}{n}}},$$
where $f$ runs over all positive definite hermitian forms in $n$
variables over $E$.  By \cite[Theorem 1]{Icaza97}, we have
$\gamma_{n,E}\leq\sigma_{n,E}$, with
$$\sigma_{n,E} = \begin{cases}
4\, \omega_{n}^{-\frac{2}{n}} & \mbox{ if $E = \mathbb Q$};\\
2\, \omega_{2n}^{-\frac{1}{n}}\, \vert d_E\vert^{\frac{1}{2}} & \mbox{ otherwise},
\end{cases}$$
where $d_E$ is the discriminant of $E$ and
$$\omega_{n} = \pi^{\frac{n}{2}}\Gamma\left(\frac{n}{2}+1\right)^{-1} = \begin{cases}
\frac{\pi^{\frac{n}{2}}}{\left(\frac{n}{2}\right)!} & \mbox{ if $n$ is even},\\
\frac{\pi^{\frac{n-1}{2}}\, 2^{n+1}\, \left(\frac{n+1}{2}\right)!}{(n+1)!} & \mbox{ if $n$ is odd}.
\end{cases}$$
It follows from Stirling's series expansion \cite[Page 253]{ww} that
$$n! = \sqrt{2\pi}\, n^{n + \frac{1}{2}}\, e^{-n}\, e^{r_n},$$
where $0 < r_n < \frac{1}{12n}$ for all positive integers $n$ (see \cite{Robbins} for an elementary proof).  It follows that
$$\sqrt{2\pi}\, n^{n + \frac{1}{2}} \, e^{-n} \leq n! \leq e\, n^{n + \frac{1}{2}}\, e^{-n},$$
for all positive integers $n$.  A straightforward calculation using these two inequalities shows that
\begin{equation}\label{hermiteconstant}
\sigma_{n, E} \leq e^{-1 + \frac{1}{n}}\,n^{1 + \frac{1}{n}}\, \vert d_E\vert^{\frac{1}{2}}, \quad \mbox{ for all $n \geq 1$}.
\end{equation}

Since $E$ will be clear from the context, we simply write $\sigma_n$ instead of $\sigma_{n,E}$.  Then,
for any positive definite hermitian form in $n$ variables over $E$, we have
\begin{equation}\label{muf}
\mu(f) \leq \sigma_n \, d(f)^{\frac{1}{n}}.
\end{equation}
For any positive integer $m$, let
\begin{equation}\label{definealpham}
\alpha(m):=\sigma_{m+1}\prod_{k=2}^{m+1}\sigma_k^{\frac{1}{k-1}}.
\end{equation}

\begin{lem} \label{hihj}
Let $f(x_1, \ldots, x_n)$ be a weakly reduced positive definite hermitian form over $E$.  Then the coefficients $h_1, \ldots, h_n$ satisfy the inequalities
\begin{equation*}
h_ih_j^{-1} \leq \alpha(j - i)
\end{equation*}
for any $1 \leq i < j \leq n$.
\end{lem}
\begin{proof}
When $E = \mathbb Q$, this lemma can be deduced from \cite[Lemma 2.4]{Schnorr} which works for weakly reduced quadratic forms, not necessary only for HKZ reduced quadratic forms.

For weakly reduced hermitian forms in general, note that, since for any $1 \leq i < j \leq n$, the form
$$g(x_i, \ldots, x_j): = f(0, \ldots, 0, x_i, \ldots, x_j, 0, \ldots, 0)$$
is weakly reduced, we may assume that $i = 1$ and $j = n$.  Then, applying the same argument as in the proof of \cite[Lemma 2.4]{Schnorr}, we obtain
$$h_1 \leq \sigma_n^{\frac{n}{n-1}} \prod_{i = 1}^{n-2} \sigma_{n-i}^{\frac{1}{n-i-1}} h_n,$$
which is exactly what we need after a simple algebraic manipulation.
\end{proof}

\begin{lem}\label{alpham}
For any positive integer $m$,
$$\alpha(m) \leq D_1 \, \vert d_E \vert^{\frac{1}{2}(1 + \Sigma(m))}\, e^{\frac{1}{2}(\ln m)^2},$$
where $D_1$ is an absolute constant and $\Sigma(m) = 1 + \frac{1}{2} + \cdots + \frac{1}{m}$.
\end{lem}
\begin{proof}
When $E = \mathbb Q$, this upper bound for $\alpha(m)$ can be found in \cite[Corollary 2.5]{Schnorr}.  In general, by virtue of \eqref{hermiteconstant}, we have
$$\alpha(m) \leq \left(e^{-1}|d_E|^{\frac{1}{2}}\right)^{1+\Sigma(m)}\,e^{\frac{1}{m+1}}(m+1)^{1+\frac{1}{m+1}}
            \prod_{k=2}^{m+1}(e^{\frac{1}{k}}k^{1+\frac{1}{k}})^{\frac{1}{k-1}}.$$
Now, let us consider
\begin{equation*}
\begin{split}
& \ln\left(e^{\frac{1}{m+1}}\,(m+1)^{1+\frac{1}{m+1}}\prod_{k=2}^{m+1}(e^{\frac{1}{k}}k^{1+\frac{1}{k}})^{\frac{1}{k-1}}\right) \\
 = & \left(1+\frac{1}{m+1}\right)\ln(m+1)+\sum_{k=2}^{m+1}\frac{1}{k-1}\ln k+\sum_{k=2}^{m+1}\frac{1}{k(k-1)}\ln k + \ln e.
\end{split}
\end{equation*}
Since $\ln(1 + m) = \ln m + O(\frac{1}{m})$, the first term becomes
$$\left(1+\frac{1}{m+1}\right)\ln(m+1)=\left(1+\frac{1}{m+1}\right)\left(\ln m+O\left(\frac{1}{m}\right)\right)\leq\ln m+ C_1,$$
where $C_1$ is an absolute constant.  For the second term, note that
$$\sum_{k=2}^{m+1}\frac{1}{k-1}\ln k=\sum_{k=1}^m\frac{\ln k}{k}+O(1)\leq\int_1^m\frac{\ln k}{k}dk+O(1)\leq \frac{(\ln m)^2}{2}+C_2,$$
where $C_2$ is another absolute constant.  At last, there is yet another absolute constant $C_3$ such that
$$\sum_{k=2}^{m+1}\frac{1}{k(k-1)}\ln k \leq \sum_{k=1}^{m}\frac{\ln k}{k^2}+O(1)\leq C_3.$$
Therefore,
$$\alpha(m) \leq D_1 \left(e^{-1}\vert d_E\vert^{\frac{1}{2}} \right)^{1 + \Sigma(m)}\, m^{1 + \frac{1}{2}\ln m},$$
where $D_1 = e^{C_1 + C_2 + C_3 + 1}$.  Furthermore, since
$$e^{-(1 + \Sigma(m))} < e^{-\ln m} = m^{-1},$$
we obtain
$$\alpha(m) \leq D_1 \, \vert d_E\vert^{\frac{1}{2}(1 + \Sigma(m))}\, e^{\frac{1}{2}(\ln m)^2}.$$
\end{proof}

%\begin{rmk}
%When $E = \mathbb Q$, Lemma \ref{hihj} is essentially  \cite[lemma 2.4]{Schnorr} which holds for weakly reduced forms.   The upper bound for $\alpha(m)$ in this case can also be found in \cite[Corollary 2.5]{Schnorr}.
%\end{rmk}

\section{Balanced HKZ reduction}

We begin this section with the general assumption that $E$ is either $\mathbb Q$ or an imaginary quadratic field, and that $\mathcal O$
is the ring of integers in $E$.  Recall that the completion of $E$ with respect to its unique archimedean prime spot is denoted by $\mathbb K$.
For the sake of convenience, in the subsequent discussion we will drop the subscript and use $\mathcal C$ and $\beta$ to denote the set $\mathcal C_E$ and the constant $\beta_E$, respectively, introduced in Section 1.

Let $n$ be a positive integer.  For any integers $i, j$ between 1 and $n$, let $E_{ij}$ be the matrix with 1 in the $(i,j)$ position and
0 elsewhere.  The set $\{E_{ij}: 1\leq i, j \leq n\}$ is a basis of $M_n(\mathbb K)$. Let $T(n)$ be the set of upper triangular matrices
in $M_n(\mathbb K)$ and $U(n)$ be the group of unipotent matrices in $T(n)$.  When the integer $n$ is clear from the context of discussion,
we will simply use $T$ and $U$ to denote $T(n)$ and $U(n)$, respectively.   For any nonnegative integer $k < n$, let $T_k$ be the subspace
of $M_n(\mathbb K)$ spanned by the matrices $E_{1,1 + k}, \ldots, E_{n-k, n}$.  Note that $T_0$ is the subspace containing all the
diagonal matrices in $M_n(\mathbb K)$, and for $k > 0$, $T_k$ is made up of matrices that have 0 everywhere outside of one of the
sub-diagonals which stays above the main diagonal.   For any integer $k \geq n$, we set $T_k$ to be zero.  Then
$$T = \bigoplus_{k \geq 0} T_k, \quad \mbox{ and } \quad T_k T_\ell \subseteq T_{k + \ell}.$$
Thus, $T$ can be viewed as a graded ring over the natural numbers.  For any nonnegative integer $k$, let
$$A_k := \bigoplus_{i\geq k} T_i$$
which is a two-sided ideal of $T$.  It is straightforward to check that $A_kA_\ell \subseteq A_{k + \ell}$ and $A_k = T_k + A_{k+1}$.
For any $X, Y\in T$ such that $X\equiv Y \mod  A_k$, there is a unique $Z\in T_k$ such that $X\equiv Y+Z \mod  A_{k+1}$.  Also of note is
that $U$ is the subset of $T$ in which every matrix is congruent to $I$ mod $A_1$.

\begin{lem}\label{xy}
For any $X\in U$, there exists  $Y\in U$ with entries in $\mathcal O$ such that $XY$ can be written as $XY=\exp(Z_1)\cdots \exp(Z_{n-1})$
where for $1 \leq k \leq n-1$, $Z_k\in T_k$ and the entries of $Z_k$ are in $\mathcal{C}$.
\end{lem}
\begin{proof}
For the sake of convenience, we set $Y_0$ and $Z_0$ to be the zero matrix in $U$. We will show by induction that for any $0 \leq k \leq n-1$, there are matrices $Y_0, Z_0, Y_1, Z_1,\ldots, Y_k,Z_k$ which satisfy the conditions:
\begin{enumerate}
\item[(a)] $Y_i, Z_i\in T_i$ for $0 \leq i \leq k$;

\item[(b)] $Y_i$ has entries in $\mathcal{O}$ and $Z_i$ has entries in $\mathcal{C}$ for any $0\leq i\leq k$;

\item[(c)] $X(I+Y_1+\cdots+Y_k)\equiv \exp(Z_1)\cdots \exp(Z_k) \mod  A_{k+1}$.
\end{enumerate}
Since $A_{n} = 0$, the matrix $Y: = I + Y_1 + \cdots + Y_{n-1}$ is what we need.

The base case $k = 0$ requires no proof; it states exactly the fact that $X\equiv I \mod  A_1$.  Suppose that we have constructed
$Y_0, Z_0, Y_1, Z_1, \ldots ,Y_{k-1}, Z_{k-1}$ which satisfy all the above conditions.  Let
$$A=X(I+Y_1+\cdots+Y_{k-1}) \quad \mbox{ and } \quad B=\exp(Z_1)\cdots \exp(Z_{k-1}).$$
It is clear that $A,B\in U$ and $A\equiv B \mod  A_k$.  Hence there is some $X_k\in T_k$ such that $A\equiv B+X_k \mod  A_{k+1}$.
By the remark made in the Introduction, there exists  $Y_k\in T_k$ with entries in $\mathcal O$ such that the entries of $Z_k:= X_k+Y_k$ are in $\mathcal{C}$. In addition,
$$A+Y_k\equiv B+X_k+Y_k\equiv B+Z_k \mod  A_{k+1}.$$
It is clear that $Y_k$ and $Z_k$ satisfy conditions (a) and (b).  It follows from $Y_k\in A_k$ and $X\equiv I\mod  A_1$
that $XY_k\equiv Y_k \mod A_{k+1}$. Similarly  $BZ_k\equiv Z_k \mod  A_{k+1}$. Thus
$$A+Y_k\equiv A+XY_k\equiv X(I+Y_1+\cdots+Y_k)\mod  A_{k+1}$$
and
$$B+Z_k\equiv B+BZ_k\equiv B(I+Z_k)\mod  A_{k+1}.$$
Note that $\exp(Z_k)=I+Z_k+Z_k^2/2!+\cdots\equiv I+Z_k\mod  A_{k+1}$, since  $Z_k^l\in A_{kl}\subseteq A_{k+1}$ for $l>1$.
Thus $B+Z_k\equiv B\exp(Z_k)\mod A_{k+1}$. It follows that
$$X(I+Y_1+\cdots+Y_k)\equiv A+Y_k\equiv B+Z_k\equiv \exp(Z_1)\cdots \exp(Z_k)\mod A_{k+1},$$
which is condition (c).  The induction is now complete.
\end{proof}

For any nonnegative integer $m$, let $c(m)$ be the coefficient of $x^m$ in the Maclaurin series of $\exp(\frac{\beta x}{1 - x})$.
For two matrices $A = (a_{ij}), B = (b_{ij}) \in M_n(\mathbb K)$, we write ``$A\preceq B$" if $\vert a_{ij} \vert \leq \vert b_{ij}\vert$
for all $1\leq i, j \leq n$.

\begin{lem}\label{dominant}
Let $A, B, C, D$ be matrices in $M_n(\mathbb K)$.  If $A \preceq B$, $C \preceq D$, and the entries of $B$ and $D$ are nonnegative real numbers,
then $AC \preceq BD$ and $A + C \preceq B + D$.
\end{lem}
\begin{proof}
This is clear.
\end{proof}

\begin{lem}\label{xy2}
For any $X\in U$, there exists $Y\in U$ with entries in $\mathcal O$ such that for $1 \leq i < j \leq n$,  the absolute values of the $(i,j)$
entries of both $XY$ and $(XY)^{-1}$ are less than or equal to $c(j-i)$.
\end{lem}
\begin{proof}
Let $X$ be a matrix in $U$, and $Y, Z_1, \ldots, Z_{n-1}$ be the matrices obtained from Lemma \ref{xy}.  Let $D=\sum_{i=1}^{n-1}E_{i,i+1}$.
A simple induction argument shows that for $1 \leq k \leq n-1$, $D^k=\sum_{i=1}^{n-k}E_{i,i+k}$ which is in $T_k$; in particular $D^n=0$.
Since $Z_k$ is also in $T_k$ and the entries of $Z_k$ are in $\mathcal{C}$, $Z_k\preceq \beta D^k.$  By Lemma \ref{dominant},
$Z_k^\ell \preceq \beta^\ell D^{k\ell}$ for every $\ell \geq 0$ and
hence $I+Z_k+\frac 12Z_k^2+\cdots\preceq I+\beta D^k+\frac
12\beta^2D^{2k}+\cdots$, i.e. $\exp(Z_k)\preceq\exp(\beta D^k)$.  Moreover,
$\exp(Z_1)\cdots \exp(Z_{n-1}) \preceq \exp(\beta D)\cdots \exp(\beta D^{n-1})$ and thus
$$XY \preceq \exp(\beta D)\cdots \exp(\beta D^{n-1}) = \exp(\beta(D+\cdots+D^{n-1}))=\exp\left(\beta D(I-D)^{-1}\right).$$
Here $D(I-D)^{-1}=D+D^2+\cdots+D^{n-1}$, since $D^k=0$ when $k\geq n$. Note that
$$\exp\left(\beta D(I-D)^{-1}\right)=\sum_{m=0}^\infty c(m) D^m=\sum_{m=0}^{n-1}c(m)D^m=\sum_{m=0}^{n-1}c(m)\sum_{i=1}^{n-m}E_{i,i+m}.$$
Therefore,  the absolute value of the $(i,i+m)$ entry of $XY$ is less than or equal to $c(m)$.  Equivalently,  the absolute value of
the $(i,j)$ entry of $XY$ is less than or equal to $c(j-i)$.

As for $(XY)^{-1}$, notice that $(XY)^{-1}=\exp(-Z_{n-1})\cdots\exp(-Z_1)$, and that $-Z_k\preceq \beta D^k$. The proof follows immediately.
\end{proof}

We will be interested in an explicit upper bound for $c(m)$, which is given in the next lemma.

\begin{lem}\label{cm}
There exists a constant $D_2$, depending only on $E$, such that
$$c(m)\leq D_2\,e^{2\sqrt{\beta m}}$$
for any $m \geq 1$.
\end{lem}
\begin{proof}
This follows from \cite{perron}; see also \cite[Page 547]{knopp}.
\end{proof}

Now, let $E$ be $\mathbb Q$ or an imaginary quadratic field with class number 1.  Let $f(x_1, \ldots, x_n)$ be a positive definite hermitian form over $E$.
By Proposition \ref{kz}, we may assume that $f(x_1, \ldots, x_n)$ is already weakly reduced, and that its associated Gram matrix is of the form $X^*HX$,
where $X \in U(n)$ and $H = \text{diag}(h_1, \ldots, h_n)$.   For any $1 \leq i < j \leq n$, there exists a function $\alpha$ defined
by \eqref{definealpham} such that $h_ih_j^{-1} \leq \alpha(j - i)$.  By Lemma \ref{alpham},
\begin{equation}\label{alphaj-i}
h_ih_j^{-1} \leq  \alpha(j - i) \leq \od (n):=D_1 \,
\vert d_E \vert^{\frac{1}{2}(1 + \Sigma(n))}\, e^{\frac{1}{2}(\ln (n))^2} ,
\end{equation}
where $D_1$ is an absolute constant and $\Sigma(n) = 1 + \frac{1}{2} + \cdots + \frac{1}{n}$.

By Lemma \ref{xy2}, we may further assume that the absolute values of the $(i, j)$ entries of {\em both} $X$ and $X^{-1}$ are bounded above by $c(j - i)$,
the coefficient of $x^{j-i}$ in the Maclaurin series of $e^{\frac{\beta x}{1 - x}}$, and by Lemma \ref{cm}
\begin{equation}\label{cj-i}
c(j - i) \leq D_2\, e^{2\sqrt{\beta (j-i)}} =: \overline{c}(j - i)
\end{equation}
where $D_2$ is constant depending only on $E$.  As a result, the absolute values of the entries of both $X$ and $X^{-1}$ are at worst in the order
of an exponential of $\sqrt{n}$.   We will say that $f(x_1, \ldots, x_n)$ is {\em balanced HKZ reduced} if it is of the form we just described.
Note that the function $\overline{c}$ is an increasing function of the natural numbers.

\section{Neighbors of hermitian lattices} \label{neighbors}

In this short section we will recall some results in the theory of neighbors of integral hermitian forms due mainly to Schiemann \cite{Sch}
which is a generalization of the theory of neighbors of integral quadratic forms developed by Kneser \cite{Kn}.  Let $E$ be
an arbitrary imaginary quadratic field.    We will adopt the geometric language of hermitian spaces and lattices in this section.
Unexplained notations and terminologies will generally be those of O'Meara \cite{OM} and Schiemann \cite{Sch}.  The readers are also referred
to \cite{Gerstein1}, \cite{Gerstein2} and \cite{Johnson66} for the local theory of hermitian lattices, and \cite{Shimura} for the global theory.
A lattice shall always mean a finitely generated $\mathcal O$-module on a positive definite hermitian space over $E$.
Whenever it is clear from the context of discussion, we will simply use $h$ to denote the hermitian map on a hermitian space.
The scale of a lattice $L$ is the fractional ideal $\mathfrak s(L)$ generated by $h(\bv, \bw)$ for all $\bv, \bw \in L$.

Given two lattices $N$ and $L$, we say that $N$ is represented by $L$ if there exists an isometry sending $N$ into $L$.
For any integer $r > 0$, let $I_r$ denote the free lattice which has an orthonormal basis.  We will be particularly interested
in lattices that can be represented by some $I_r$.  If a free lattice $N$ is one of these lattices and $\{\bv_1, \ldots, \bv_n\}$
is a basis for $N$, then the hermitian form $\sum_{1\leq i, j \leq n} h(\bv_i, \bv_j)x_ix_j^*$ is a sum of the norms of $r$ linear forms over $\mathcal O$.

Let $L$ be an integral lattice (i.e. $\mathfrak s(L) \subseteq \mathcal O$) on a hermitian space $V$ over $E$ and $\p$ be a nonzero
prime ideal of $\mathcal O$ which does not divide the volume of $L$.  An integral lattice $M$ on $V$ is called a $\p$-neighbor of
$L$ if $M/(L\cap M) \cong \mathcal O/\p$ and $L/(L \cap M) \cong \mathcal O/\p^*$.  It is clear from the definition that
$\p M \subseteq L$ if $M$ is a $\p$-neighbor of $L$.  Let $\norm(L, \p)$ be the set of lattices $M$ on $V$ such that there exist
isometry $\phi$ of $V$ and a sequence of lattices $L_0 = L, L_1, \ldots, L_k = \phi(M)$ such that  $L_{i+1}$ is a $\p$-neighbor
of $L_i$ for $i = 0, \ldots, k-1$.   From \cite[Corollary 2.7]{Sch}, we know that if the rank of $L$ is at least 3, then the
special genus of $L$ (see \cite[Definition 1.7]{Sch}), $\gen^0(L)$, is contained in $\norm(L, \p)$.  It then follows
from \cite[Lemma 2.8]{Sch} that when $m\geq 3$ is odd,
$$\gen^0(I_m) \subseteq \norm(I_m, \p) \subseteq \gen(I_m).$$
This implies that when $m \geq 3$ is odd, any $M \in \gen^0(L)$ must have an isometric copy which is at most $\mathfrak h_m$
steps away from $I_m$ in $\norm(I_m, \p)$.  Here $\mathfrak h_m$ is the class number of $I_m$.  As a result, $\p^{\mathfrak h_m}M$
must be represented by $I_m$.

Let $\sigma$ be the positive integer defined by
\begin{equation}\label{h3}
\sigma = \begin{cases}
2\mathfrak h_3 & \mbox{ if 2 is inert in $E$};\\
\mathfrak h_3 & \mbox{ otherwise}.
\end{cases}
\end{equation}
An explicit formula for $\mathfrak h_3$ can be found in \cite{Hashimoto}.

\begin{lem}\label{rank2}
Let $N$ be a lattice of rank $2$.  If $\mathfrak s(N) \subseteq 2^\sigma\mathcal O$ where $\sigma$ is defined as in \eqref{h3}, then $N$
is represented by $I_3$.
\end{lem}
\begin{proof}
Let $\p$ be a prime ideal of $\mathcal O$ lying above 2.  Then the lattice $\p^{-\sigma}N$ is integral, and it follows from by the local
theory of hermitian lattices that $\p^{-\sigma}N$ is represented by some lattice $K$ in $\gen^0(I_3)$.  Since $\p^\sigma K$ is represented
by $I_3$, $N$ is represented by $I_3$ as well.
\end{proof}

\section{Main results} \label{mainresults}

Let $E$ be $\mathbb Q$ or an arbitrary imaginary quadratic field, $\mathcal O$ be its ring of integers, and $\mathbb K$ be the completion of
$E$ with respect to its unique archimedean prime spot.   We define an element $\omega$ of $E$ as follows.  If $E = \mathbb Q$, let $\omega$ be 1.
Otherwise, if $E = \mathbb Q(\sqrt{-\ell})$, $\ell$ squarefree, let
$$\omega = \begin{cases}
\sqrt{-\ell} & \mbox{ if $\ell \equiv 1, 2$ mod 4};\\
\frac{1 + \sqrt{-\ell}}{2} & \mbox{ if $\ell \equiv 3$ mod 4}.
\end{cases}$$

A matrix $H$ with entries in $\mathbb K$ is called hermitian if $H = H^*$.  Every hermitian matrix with entries in $\mathcal O$ is the
Gram matrix of an integral hermitian form over $\mathcal O$.   Let $A$ be a hermitian matrix with entries in $\mathcal O$.  We say that
$A$ is represented by $I_r$ if the hermitian form associated to $A$ is the sum of the norms of $r$ linear forms over $\mathcal O$.

\begin{lem} \label{apluss}
Let $n \geq 2$ be a positive integer, $A=\mathrm{diag}(a_1,...,a_n)$ be a diagonal matrix in  $M_n(\mathbb{R})$,  and $S = (s_{ij})$ be
a hermitian matrix in $M_n(\mathbb K)$.  Suppose that for each $1 \leq i \leq n$, $a_i = \sum_{j=1}^n t_{ij}$ with $t_{ij} > 0$ and
$t_{ij}t_{ji} \geq \vert s_{ij}\vert^2$ for all $j$.
\begin{enumerate}
\item The hermitian matrix $A+S$ is positive semidefinite.

\item Suppose that $t_{ij}, s_{ij}$ are in $\mathcal{O}$ for all $1 \leq i, j \leq n$,  and that
$t_{ii}+s_{ii}\geq 2^{\sigma}(n-1)(\textnormal{N}_{E/\mathbb Q}(\omega)+4)$
for all $i$, where $\sigma$ is defined by \eqref{h3} if $E \neq \mathbb Q$ and  $\sigma = 0$ otherwise.  Then $A+S$ is represented by $I_{2^{\sigma+2}n^2}$.
\end{enumerate}
\end{lem}
\begin{proof}
\noindent (1): For each $i$, we have  $t_{ii}t_{ii}\geq |s_{ii}|^2$; hence $t_{ii}\geq |s_{ii}|$. It follows that
$$a_i+s_{ii}=\sum_jt_{ij}+s_{ii}=\sum_{j\neq i}t_{ij}+(s_{ii}+t_{ii})\geq \sum_{j\neq i}t_{ij}.$$
Consequently, we can write $a_i+s_{ii}$ as $\sum_{j\neq i}t'_{ij}$, where $t'_{ij}\geq t_{ij}$.  Furthermore, if $t_{ij}$ and $s_{ii}$
are rational integers, then $t'_{ij}$ are  rational integers as well. Therefore
\begin{equation}
\begin{split}
A+S&=\sum_ia_iE_{ii}+\sum_{1\leq i,j\leq n}s_{ij}E_{ij}\\
   &=\sum_i(a_i+s_{ii})E_{ii}+\sum_{j\neq i}s_{ij}E_{ij}\\
   &=\sum_{j\neq i}t'_{ij}E_{ii}+\sum_{j\neq i}s_{ij}E_{ij}\\
   &=\sum_{i<j}(t'_{ij}E_{ii}+t'_{ji}E_{jj}+s_{ij}E_{ij}+s_{ji}E_{ji}).\nonumber
\end{split}
\end{equation}
Note that for $i < j$, $t'_{ij}E_{ii}+t'_{ji}E_{jj}+s_{ij}E_{ij}+s_{ji}E_{ji}$ is an $n\times n$ positive semidefinite hermitian matrix of
rank at most 2 because $t'_{ij}t'_{ji}\geq t_{ij}t_{ji}\geq \vert s_{ij}\vert^2 =s_{ij}s_{ji}$. Hence $A+S$ is a positive semidefinite hermitian matrix.

\noindent (2): If $E = \mathbb Q$, we know that $I_5$ is 2-universal, i.e., it represents all positive semidefinite integral binary
quadratic forms over $\mathbb Z$.  From part (1), each
$$t'_{ij}E_{ii}+t'_{ji}E_{jj}+s_{ij}E_{ij}+s_{ji}E_{ji}$$
is positive semidefinite and integral, and hence it must be represented by $I_5$.  It follows that
$A  + S$ is represented by $I_{\frac{5}{2}n^2 - \frac{5}{2}n}$ in this case.

For the rest of the proof, let $E = \mathbb Q(\sqrt{-\ell})$, $\ell$ squarefree,  be an imaginary quadratic field.    Suppose that
$A = \text{diag}(a_1, \ldots, a_n)$ and $S = (s_{ij})$ are matrices which satisfy the given conditions.  Let
$$\mathcal P = \{a + b\omega : a , b \in \mathbb Z, \, -(2^{\sigma+1} - 1) \leq a, b \leq 2^{\sigma+1} - 1\}.$$
For any $1 \leq i < j \leq n$, there exist $p_{ij} \in \mathcal P$ such that $p_{ij}\equiv s_{ij}$ mod $2^\sigma\mathcal O$, and
\begin{enumerate}
\item[(a)] $s_{ij} - p_{ij} = 2^\sigma q_{ij}$ for some $q_{ij} \in \mathcal O$,

\item[(b)] $\vert s_{ij}\vert^2 \geq \vert 2^\sigma q_{ij} \vert^2$,

\item[(c)] $p_{ij} = p_{ji}^*$ and $q_{ij} = q_{ji}^*$.
\end{enumerate}

For any $p_{ij}=a+b\omega \in \mathcal P$ where $i<j$, let $n_{ij}$ be the positive integer
$\vert a \vert+\lfloor\frac{\vert b\vert}{2}\rfloor \textnormal{N}_{E/\mathbb Q}(\omega)+\lceil\frac{\vert b\vert}{2}\rceil$ and $n_{ji}$ be the positive integer $\vert a \vert+\lfloor\frac{\vert b\vert}{2}\rfloor+\lceil\frac{\vert b\vert}{2}\rceil\textnormal{N}_{E/\mathbb Q}(\omega)$.  Then,
\begin{equation} \label{2x2}
\begin{pmatrix}
      n_{ij} & p_{ij} \\
      p_{ji} & n_{ji}
\end{pmatrix}
     =  TT^*,
\end{equation}
where $T$ is the following $2\times (\vert a \vert +\lfloor\frac{\vert b \vert}{2}\rfloor+\lceil\frac{\vert b\vert}{2}\rceil)$ matrix
$$T: = \left(
      \begin{array}{ccccccccc}
        \epsilon_a & \cdots & \epsilon_a & \epsilon_b\omega & \cdots & \epsilon_b\omega & \epsilon_b & \cdots & \epsilon_b \\
        1 & \cdots & 1 & 1 & \cdots & 1 & \omega^* & \cdots & \omega^* \\
      \end{array}
    \right).
$$
Here, $\epsilon_a$ and $\epsilon_b$ are the signs of the integers $a$ and $b$, respectively.  In particular, the $2\times 2$ matrix in \eqref{2x2}
is  represented by $I_{\vert a\vert +\vert b\vert}$, and hence it is represented by $I_{2^{\sigma+2}-2}$ as well.

Let $\tilde{n}$ be the largest of the $n_{ij}$'s.  Note that $\tilde{n} \leq 2^{\sigma + 1} + 2^\sigma + 2^\sigma \textnormal{N}_{E/\mathbb Q}(\omega)$. Fix an index $i$.  By hypothesis,
$$a_i = \sum_{j = 1}^n t_{ij}$$
where each $t_{ij}$ are positive integers.  Then, for any $j \neq i$, $t_{ij} + \tilde{n} = n_{ij} + r_{ij}$ for some positive
integer $r_{ij}$, and we write $r_{ij}=2^\sigma t'_{ij}-\delta_{ij}$ for some integers $t'_{ij},\delta_{ij}$ with $0\leq\delta_{ij}<2^\sigma$.  Hence,
\begin{eqnarray*}
a_i+s_{ii}=t_{ii} + s_{ii} + \sum_{j\neq i} t_{ij} & = & t_{ii} + s_{ii} - \tilde{n}(n-1) + \sum_{j\neq i}(t_{ij} + \tilde{n})\\
    & = & t_{ii} + s_{ii} - \tilde{n}(n-1) + \sum_{j \neq i} (n_{ij} + r_{ij})\\
    & = & b_i + \sum_{j\neq i} n_{ij} + \sum_{j\neq i} 2^\sigma t_{ij}'
\end{eqnarray*}
where
\begin{eqnarray*}
b_i & = & t_{ii}+s_{ii}-\tilde n(n-1)-\sum_{j\neq i}\delta_{ij} \\
    & \geq & t_{ii} + s_{ii} - \left((n-1)(2^{\sigma +1}+2^\sigma +2^\sigma\textnormal{N}_{E/{\mathbb Q}}(\omega ))+(n-1)2^\sigma \right) \\
    & = & t_{ii} + s_{ii} - 2^\sigma (n-1)(\textnormal{N}_{E/{\mathbb  Q}}(\omega )+4) \\
    & \geq & 0.
\end{eqnarray*}
Also $2^\sigma t_{ij}' \geq r_{ij} = t_{ij} + \tilde{n} - n_{ij} \geq t_{ij}$ for $j \neq i$.  Now we can write
\begin{equation*}
\begin{split}
A+S&=\sum_ia_iE_{ii}+\sum_{1\leq i,j\leq n}s_{ij}E_{ij}\\
   &=\sum_i(t_{ii}+s_{ii})E_{ii}+\sum_{j\neq i}t_{ij}E_{ii}+\sum_{j\neq i}2^{\sigma}q_{ij}E_{ij}+\sum_{j\neq i}p_{ij}E_{ij}\\
   &=\sum_ib_iE_{ii}+\sum_{j\neq i}2^{\sigma}t'_{ij}E_{ii}+\sum_{j\neq i}2^{\sigma}q_{ij}E_{ij}+\sum_{j\neq i}n_{ij}E_{ii}+\sum_{j\neq i}p_{ij}E_{ij}\\
   &=\mathrm{diag}(b_1,...,b_n)+\sum_{i<j}(2^\sigma t'_{ij}E_{ii}+ 2^\sigma q_{ij}E_{ij} + 2^\sigma q_{ji}E_{ji} + 2^\sigma t_{ji}'E_{jj}) \\
    & \quad  + \sum_{i < j} (n_{ij}E_{ii} + p_{ij}E_{ij} + p_{ji}E_{ji} + n_{ji}E_{jj}).\nonumber
\end{split}
\end{equation*}
Each of the $\frac{n(n-1)}{2}$ hermitian matrices $(n_{ij}E_{ii} + p_{ij}E_{ij} + p_{ji}E_{ji} + n_{ji}E_{jj})$ is represented by $I_{2^{\sigma+2}-2}$
as indicated in \eqref{2x2}. The diagonal matrix $\text{diag}(b_1, \ldots, b_n)$ is represented by $I_{4n}$ by Lagrange's Four-Square Theorem.
Each of the $\frac{n(n-1)}{2}$ matrices $(2^\sigma t_{ij}'E_{ii} + 2^\sigma q_{ij}E_{ij} + 2^\sigma q_{ji}E_{ji} + 2^\sigma t_{ji}'E_{jj})$
is an $n\times n$ positive semidefinite hermitian matrix of rank at most 2 with entries divisible by $2^\sigma$ and hence it is represented
by $I_3$ by Lemma \ref{rank2}.  The proof is now completed since $4n+\frac{1}{2}n(n-1)(2^{\sigma+2} + 1) \leq 2^{\sigma+2}n^2$.
\end{proof}

\begin{cor}\label{smallcor}
If $S = (s_{ij})$ is a hermitian matrix with $|s_{ij}|\leq \frac{1}{n}$, then $I_n+S$ is positive semidefinite.
\end{cor}
\begin{proof}
We apply Lemma \ref{apluss}(a) to the case when $A$ is the $n\times n$ identity matrix.  Notice that for each $1 \leq i \leq n$,
$1 = \sum_{j = 1}^n t_{ij}$ with $t_{ij}=\frac{1}{n}$ and $t_{ij}t_{ji}=\frac{1}{n^2}\geq \vert s_{ij}\vert^2$.
\end{proof}

The following proposition is a consequence of all the results we have accumulated thus far.  It will be crucial in deriving our upper
bounds on $g_{\mathcal O}^*(n)$.

\begin{prop}\label{prop1}
Let $E$ be $\mathbb Q$ or an imaginary quadratic field with class number 1.  There is a function
$$G_E(n)=D_3\,  |d_E|^{2(1+\Sigma(n))}\, n^{10}e^{(4+4\sqrt{2})\sqrt{\beta n}+2(\ln n)^2},$$
where $D_3$ is a constant depending only on $E$, such that every positive definite integral hermitian form $f$ in $n \geq 2$
variables with $\mu(f)\geq G_E(n)$ can be represented by $I_{2^{\sigma+2}n^2 + n}$.
\end{prop}
\begin{proof}
The strategy of the proof is as follows.  Let $f$ be a positive definite hermitian form in $n$ variables.  We may assume that $f$ is
already balanced HKZ reduced.  Let $M$ be the associated Gram matrix.  We take a diagonal matrix $A = \text{diag}(a_1, \ldots, a_n)$
with all the $a_i$ as large as possible such that $M - A$ remains positive semidefinite.  Then we take $P \in M_n(\mathcal O)$ such
that $P^*P$ approximates $M - A$ well.  Write $M - A$ as $P^*P + S$, or equivalently, $M = A + S + P^*P$.  We will show that, for a suitable
choice of $D_3$, if $\mu(f) \geq G_E(n)$, then $A$ and $S$ satisfy the conditions in Lemma \ref{apluss}.  As a result, $A + S$ will be represented by
$I_{2^{\sigma + 2}n^2}$.  Since $P^*P$ is clearly represented by $I_n$, $M$ will be represented by $I_{2^{\sigma + 2}n^2 + n}$.

For the rest of the proof, we will take $D_3$ large enough so that
\begin{equation} \label{d3}
D_3 \geq \max\{\beta^2,\,\, D_1^2\,D_2^2,\,\, 144\beta^2D_1^4\,D_2^6,\,\, 2^{\sigma+2}(\textnormal{N}_{E/\mathbb Q}(\omega) + 4)\,D_1^2\,D_2^2\}
\end{equation}
where $D_1$ and $D_2$ are the constants appeared in \eqref{alphaj-i} and \eqref{cj-i} respectively.

Since $f$ is balanced HKZ reduced, $M = H[X]: = X^*HX$, where $H$ is a diagonal matrix $\text{diag}(h_1, \ldots, h_n)$, $X \in U(n)$, and $h_1 = \mu(f)$. By \eqref{alphaj-i}, $\od (n)h_i\geq h_1$ for all $i>1$. Hence the hypothesis implies that
\begin{equation} \label{GEn}
\od (n)h_i\geq h_1=\mu (f)\geq G_E(n) = D_3\,  |d_E|^{2(1+\Sigma(n))}\, n^{10}e^{((4+4\sqrt{2})\sqrt{\beta n}+2(\ln n)^2}.
\end{equation}

Let $\sqrt{H}$ be $\text{diag}(\sqrt{h_1}, \ldots, \sqrt{h_n})$, the ``square root" of $H$.  Then $M = I_n[\sqrt{H}X]$, and hence
$$M - A = \left(I_n - A[X^{-1}\sqrt{H}^{-1}]\right)[\sqrt{H}X].$$
This shows that $M - A$ is positive semidefinite if and only if $I_n - A[X^{-1}\sqrt{H}^{-1}]$ is.

For any $1 \leq k \leq n$, let
$$a_k: = \left\lfloor \frac{1}{n^2}\, \overline{\alpha}(n)^{-1}\, \overline{c}(n - k)^{-2} h_k \right\rfloor.$$
Let $y_{ij}$ and $b_{ij}$ be the $(i,j)$ entries of $X^{-1}$ and $A[X^{-1}]$, respectively.  Suppose that $1 \leq k \leq \min\{i, j\}$.  By \eqref{alphaj-i}, we have $\overline\alpha(n)^{-1}h_k\leq\sqrt{h_i}\sqrt{h_j}$.  At the same time, \eqref{cj-i} shows that  $|y_{ki}|\leq\overline c(i - k)\leq \overline c(n-k)$ for each $i$.  Therefore,
$$\vert b_{ij} \vert \leq \sum_{k = 1}^{\min\{i, j\}} \vert a_k\vert\,
\vert y_{ki}^* \vert\, \vert y_{kj}\vert \leq
\frac{1}{n}\sqrt{h_i}\sqrt{h_j}.$$
The $(i,j)$ entry of $A[X^{-1}\sqrt{H}^{-1}]$ is $\sqrt{h_i}^{-1}b_{ij} \sqrt{h_j}^{-1}$, with absolute value $\leq \frac {1}{n}$. Thus, by Corollary \ref{smallcor},
$I_n - A[X^{-1}\sqrt{H}^{-1}]$ and $M - A$ are positive semidefinite.

By the Gram-Schmidt process, we can find an upper triangular matrix $N
= (n_{ij})$ such that
$$I_n - A[X^{-1}\sqrt{H}^{-1}] = N^* N.$$
Now, $I_n - N^*N = A[X^{-1}\sqrt{H}^{-1}]$ is positive semidefinite
whose $(j,j)$ entry is $1 - \sum_{i\leq j}\vert n_{ij}\vert^2$.
This implies that $\vert n_{ij} \vert \leq 1$ for any $i \leq j$.

Let $W = (w_{ij})$ be the matrix $N\sqrt{H} X$, such that $W^*W =
N^*N[\sqrt{H}X] = M - A$.  Note that $W$ is an upper triangular matrix in $M_n(\mathbb K)$.
Let $P = (p_{ij})$ be an upper triangular matrix in $M_n(\mathcal O)$ such that the matrix $Q = (q_{ij}) = W - P$ has entries in $\mathcal C$.
Then
$$P^*P=(W-Q)^*(W-Q)=M-A-Q^*W-W^*Q+Q^*Q$$
and $M = P^*P + A + S$, where $S := Q^*W + W^*Q - Q^*Q$ which is a hermitian matrix with entries in $\mathcal O$.

The next step is to estimate the size of the entries of $S$. Let $x_{ij}$ be the $(i,j)$ entry of $X$.  Since $W = N\sqrt{H} X$, by \eqref{alphaj-i}
and \eqref{cj-i} we have
$$\vert w_{ij}\vert = \left \vert \sum_{k = i}^jn_{ik}\sqrt{h_k}\,x_{kj}\right\vert \leq \sum_{k = i}^j \sqrt{h_k}\, \vert x_{kj} \vert \leq
 n\, \oc(j)\, (\od(n)h_j)^{\frac{1}{2}},$$
for any $1\leq i \leq j \leq n$.  Then the $(i,j)$ entry of $Q^*W$ is
$$\left\vert \sum_{k = 1}^{\min\{i,j\}} \overline{q}_{ki}w_{kj}\right \vert \leq  n^2\, \beta\, \oc(j)\, (\od(n)h_j)^{\frac{1}{2}}.$$
Since $|q_{ij}|\leq\beta$, the absolute values of the entries of $Q^*Q$ are bounded above by $\beta^2n$.   By \eqref{d3} and \eqref{GEn},
we have $n\oc(j)(\od(n)h_j)^{\frac 12}\geq\sqrt{G_E(n)} \geq \beta$.   This shows that for any $1 \leq i \leq j \leq n$,
\begin{equation} \label{sij}
\vert s_{ij} \vert \leq  2n^2\, \beta\,  \oc(j)\, (\od(n)h_j)^{\frac{1}{2}}+n\beta^2 \leq 3n^2\, \beta\, \oc(j)\, (\od(n)h_j)^{\frac{1}{2}}.
\end{equation}

For each $1\leq i \leq n$, we can write $a_i = \sum_{j = 1}^n t_{ij}$. By \eqref{d3} and \eqref{GEn},
$$\od(n) h_i \geq D_1^2\,D_2^2\, n^3\,\vert d_E \vert^{1 + \Sigma(n)}\, e^{4\sqrt{\beta n}} \geq n^3\, \od(n)^2 \oc(n)^2 \geq n^3\od(n)^2\oc(n - i)^2.$$
Therefore, $\frac{1}{n^3} \od(n)^{-1}\, \oc(n - i)^{-2}\, h_i \geq 1$, and thus
$$t_{ij} \geq  \left\lfloor \frac{a_i}{n} \right\rfloor = \left\lfloor\frac{1}{n^3} \od(n)^{-1}\, \oc(n - i)^{-2}\, h_i \right\rfloor\geq
\frac{1}{2n^3} \od(n)^{-1}\, \oc(n - i)^{-2}\, h_i.$$
By \eqref{d3}, \eqref{GEn}, and \eqref{sij}, we have
\begin{eqnarray*}
t_{ij}t_{ji} & \geq & \frac {1}{4n^6}\,\od(n)^{-2}\,\oc(n-i)^{-2}\, \oc(n-j)^{-2}\, h_ih_j \\
    & \geq & 4\vert s_{ij}\vert^2 \,\frac{\od(n) h_i}{144\, n^{10} \beta^2\, \oc(j)^2\, \oc(n-j)^2\, \oc(n-i)^2\, \od(n)^4} \\
    & \geq & 4\vert s_{ij} \vert^2 \,\frac{D_1^4\, D_2^6 |d_E|^{2(1+\Sigma(n))}e^{(4+4\sqrt 2)\sqrt{\beta n}+2(\ln n)^2}}{\oc(n)^2\,\od(n)^4\, \oc(j)^2\,\oc(n-j)^2}\\
    & \geq & 4\vert s_{ij} \vert^2 \, \frac{D_2^4\, e^{4\sqrt{2\beta n}}}{\oc(j)^2\,\oc(n-j)^2}.
\end{eqnarray*}
But by \eqref{cj-i}, $\oc(n-j)^2\,\oc(j)^2\leq D_2^4\,e^{4\sqrt{\beta (n-j)}+ 4\sqrt{\beta j}}$, and the right hand side of this inequality is maximized at $j = \frac{n}{2}$.  Therefore, $$\oc(n-j)^2\,\oc(j)^2\leq D_2^4\,e^{4\sqrt{2\beta n}},$$
and hence
\begin{equation} \label{tijsij}
t_{ij}t_{ji} \geq 4\vert s_{ij} \vert^2 \geq \vert s_{ij} \vert^2.
\end{equation}

We may also conclude from the first inequality in \eqref{tijsij} that $|s_{ii}|\leq\frac{1}{2}t_{ii}$.   Hence
\begin{eqnarray*}
t_{ii}+s_{ii}\geq\frac {1}{2}t_{ii}\geq\frac {1}{4n^3}\od(n)^{-1}\oc(n-i)^{-2}h_i & \geq & \frac {1}{4n^3}\od(n)^{-1}\oc(n)^{-2}h_i\\
    & = & \frac{1}{4n^3} \od(n)^{-2}\oc(n)^{-2}\, \od(n)h_i,
\end{eqnarray*}
which is at least $2^{\sigma}(n-1)(\textnormal{N}_{E/\mathbb Q}(\omega) + 4)$ because
\begin{eqnarray*}
\od(n)h_i\geq G_E(n) & \geq & 2^{\sigma+2}(n-1)(\textnormal{N}_{E/\mathbb Q}(\omega) + 4)n^3D_1^2D_2^2|d_E|^{1+\Sigma(n)}e^{4\sqrt{\beta n}+(\ln
  n)^2}\\
    & = & 2^{\sigma+2}(n-1)(\textnormal{N}_{E/\mathbb Q}(\omega) + 4)n^3\od(n)^2\oc(n)^2.
\end{eqnarray*}
We may now apply Lemma \ref{apluss}(2) to complete the proof.
\end{proof}

\begin{prop}\label{prop2}
Let $E$ be $\mathbb Q$ or an imaginary quadratic field with class number 1.  For any positive integer $n\geq 2$,
$$g_{\mathcal O}^*(n) \leq \max \left\{g_{\mathcal O}^*(n-1) + G_E(n), \quad 2^{\sigma + 2}n^2 + n \right\}$$
where $G_E(n)$ is the function defined in Proposition \ref{prop1}.
\end{prop}
\begin{proof}
Let $f(x_1, \ldots, x_n)$ be a hermitian form which is represented by $I_r$ for some positive integer $r$.   If $\mu(f) \geq G_E(n)$, then by
Proposition \ref{prop1} $f$ is represented by $I_{2^{\sigma + 2}n^2 + n}$.

Suppose that $\mu(f) < G_E(n)$.  We may assume that $f$ is balanced HKZ reduced and $\mu(f) = f(1, 0, \ldots, 0)$.  There are $r$ linear forms
$\ell_1(x_1, \ldots, x_n), \ldots, \ell_r(x_1, \ldots, x_n)$ over $\mathcal O$ such that
$$f(x_1, \ldots, x_n) = \sum_{j = 1}^r \norm(\ell_j(x_1, \ldots, x_n)).$$
If $b_1, \ldots, b_r$ are the coefficients of $x_1$ in $\ell_1, \ldots, \ell_r$ respectively, then at most $\lfloor G_E(n) \rfloor$ of them are nonzero.
Without loss of generality, we can write
$$f(x_1, \ldots, x_n) = \sum_{j = 1}^{\lfloor G_E(n) \rfloor} \norm(\ell_j(x_1, \ldots, x_n)) +
\sum_{j = \lfloor G_E(n) \rfloor + 1}^r \norm(\ell_j(x_2, \ldots, x_n)).$$
The second sum is a hermitian form in $n-1$ variables represented by $I_{r - \lfloor G_E(n) \rfloor}$, and hence it is represented by
$I_{g^*_{\mathcal O}(n-1)}$. The proposition follows immediately.
\end{proof}

We are now ready to prove Theorem \ref{main}.

\begin{proof}[Proof of Theorem \ref{main}]
By increasing, if necessary, the constant $D_3$ from the definition of
$G_E(n)$, we may assume that $G_E(n) > 2^{\sigma + 2}n^2 + n$ for all $n \geq 2$.  Then, by Proposition \ref{prop2},
$$g_{\mathcal O}^*(n) \leq \sum_{j = 2}^n G_E(j) + g_{\mathcal O}^*(1), \quad \mbox{ for } n \geq 2.$$
It is clear that $g^*_{\mathcal O}(1) \leq 4$ by virtue of Lagrange's
Four-Square Theorem.  As a result,
$$g_{\mathcal O}^*(n) \leq nG_E(n)=D_3\,  |d_E|^{2(1+\Sigma(n))}\, n^{11}e^{(4+4\sqrt{2})\sqrt{\beta n}+2(\ln n)^2}.$$
It is easy to show that $\Sigma(n)\leq 1+\ln n$ for all integers $n \geq 1$. Therefore, for any $\varepsilon >0$ we have
$|d_E|^{2(1+\Sigma(n))}\, n^{11}e^{2(\ln n)^2}=O(e^{\varepsilon\sqrt n})$.
It follows that
$$g_{\mathcal O}^*(n)=O\left( e^{(k+\varepsilon)\sqrt n}\right),$$
where $k=(4+4\sqrt 2)\sqrt\beta$.
\end{proof}

\section{$s$-integrable lattices} \label{conway}

In \cite{cs}, Conway and Sloane introduce the following notion of $s$-integrable $\z$-lattices.  Let $s$ be a positive integer.  A $\z$-lattice $L$ is called $s$-integrable if $\sqrt{s}L$ can be represented by a sum of squares.  Define
$$\phi(s): = \min\{n : \exists\, \mbox{ $\z$-lattice $L$ of rank $n$ such that } \sqrt{s}L \not \in \Sigma_\z(n) \}.$$
Thus, $\phi(s)$ is the smallest positive integer $n$ for which there exists a positive definite integral quadratic form $f$ in $n$ variables such that $sf$ cannot be written as a sum of squares of linear forms with integral coefficients.  The results of Ko and Mordell  mentioned in Section 1 imply that $\phi(1) = 6$.  It is known \cite{cs} that $\phi(2) = 12$ and $\phi(3) = 14$.  Upper and lower bounds for $\phi(s)$ for all $s$ can also be found in \cite{cs}.  In particular,
\begin{equation} \label{csbound}
\phi(s) \geq 2 \left(\frac{\ln\ln s}{\ln\ln\ln s}(1 + o(1)) \right).
\end{equation}
Kim and Oh later \cite{Kim-Oh05} improve this lower bound to
\begin{equation} \label{kobound}
\phi(s) \geq \frac{\ln s}{8 \ln\ln s} \quad \mbox{ for large } s.
\end{equation}
It is clear that how one can define $s$-integrable hermitian lattices and generalize the definition of $\phi(s)$ for hermitian lattices.  To this end, let $E$ be either $\mathbb Q$ or an imaginary quadratic field, and $\mathcal O$ be its ring of integers.  We define
$$\phi_{\mathcal O}^*(s) : = \min \{n: \exists\, \mbox{ $\mathcal O$-lattice $L$ of rank $n$ such that } \sqrt{s}L \not \in \Sigma^*_{\mathcal O}(n) \}.$$
Clearly, $\phi_{\z}^*(s) = \phi(s)$ by definition.   We now apply our results to obtain a lower bound for $\phi_{\mathcal O}^*(s)$ when $E$ has class number 1, which is an improvement of both \eqref{csbound} and \eqref{kobound} in the case when $E = \mathbb Q$.

\begin{thm} \label{sintegrable}
Let $E$ be $\mathbb Q$ or an imaginary quadratic field with class number 1.  Then
$$\phi_{\mathcal O}^*(s) \geq \left(\frac{\ln s}{(4 + 4\sqrt{2})\sqrt{\beta_E}} \right)^2 (1 + o(1)),$$
where $\beta_E$ is the constant defined in \eqref{betaE}.
\end{thm}
\begin{proof}
If $f$ is a positive definite hermitian form in $n$ variables which is not $s$-integrable, then $sf$ cannot be written as a sum of norms.  Then, by Proposition \ref{prop1}, we have
$$s \leq \mu(sf) \leq G_E(n) = O\left(e^{(k_E + \epsilon)\sqrt{n}}\right)\quad \mbox{ for all } \epsilon > 0,$$
where $k_E = (4 + 4\sqrt{2})\sqrt{\beta_E}$.  This can also be written as $s \leq e^{k_E\sqrt{n}(1 + o(1))}$ or $\ln s \leq k_E(\sqrt{n}(1 + o(1))$, whence
$n \geq \left(\frac{\ln s}{k_E} \right)^2 (1 + o(1))$.  The theorem now follows from the definition of $\phi_{\mathcal O}^*(s)$.
\end{proof}

\appendix

\section{Finiteness of $g_{\mathcal O}^*(n)$}

Let $E/F$ be a CM extension, $*$ be the nontrivial automorphism in $\text{Gal}(E/F)$, and $\mathcal O$ be the ring of integers in $E$.
Every hermitian form discussed in this appendix is defined with respect to $*$.  We will once again adopt the geometric language of hermitian
spaces and lattices used in Section \ref{neighbors}.    Let $\Sigma^*_{\mathcal O}(n)$ be the set of integral hermitian lattices which are
represented by sums of norms, and define
$$g_{\mathcal O}^*(n) = \min\{g : \mbox{ every hermitian lattice in } \Sigma_{\mathcal O}^*(n) \mbox{ is represented by } I_g\}.$$
This definition matches the one we introduced in Section \ref{introduction} since every lattice is free when $\mathcal O$ is a PID, and
isometry classes of free lattices correspond to equivalence classes of hermitian forms.  We will give a brief explanation of why $g_{\mathcal O}^*(n)$
is finite.   The argument is essentially the same as the one used in \cite[Proposition 6]{Icaza96} to show that $g_{\mathcal O}(n)$ is finite when
$\mathcal O$ is the ring of integers of a totally real number field.\footnote{There is a couple of small missteps in the proof of \cite[Proposition 6]{Icaza96}.
First, Humbert's reduction does not guarantee that the minimum of a positive definite integral quadratic form is attained at a unimodular vector when
$\mathcal O$ is not a PID.  Second, $g_{\mathcal O}(1)$ is not bounded above by 5 as claimed; as a matter of fact $g_{\mathcal O}(1)$, although it is
always finite, can be made arbitrary large, see \cite{Scharlau79}.}

Suppose that $N$ is a lattice of rank $n$ which is represented by $I_r$ for some $r$.  We may assume that $N$ is positive definite.
Let $\mu(N)$ be $\min\{\text{Tr}_{F/\mathbb Q}(h(x)): 0 \neq x \in N\}$. By the local representation theory of hermitian lattices
(\cite{Gerstein2} and \cite{Johnson66}), we see that $N$ is represented by the genus of $I_{2n + 1}$.  Let $c = c(2n+1)$ be the constant in
\cite[Theorem 2.12]{JP94} for the lattice $I_{2n+1}$.  If $\mu(N) \geq c$, then $N$ is already represented by $I_{2n+1}$.

Suppose that $\mu(N) < c$.  Then there exists nonzero vector $v \in N$ such that $0 < \text{Tr}_{F/\mathbb Q}(h(v)) < c$.
Let $\sigma$ be a representation of $N$ by $I_r$, and for any $1\leq i \leq r$, let $\phi_i$ be the $i$-th coordinate function of $I_r$.
Since $\text{Tr}_{F/\mathbb Q}(h(v)) = \sum_{i=1}^r \text{Tr}_{F/\mathbb Q}(\phi_i(v)^*\phi_i(v))$ and all
$\text{Tr}_{F/\mathbb Q}(\phi_i(v)^*\phi_i(v))$ are nonnegative rational integers, there are at most $\lfloor c \rfloor$ of the $\phi_i(v)$ are nonzero.
Without loss of generality, we may assume that $\phi_i(v) = 0$ for $\lfloor c \rfloor + 1 \leq i \leq r$.
Let $\sigma_1 = (\phi_1, \ldots, \phi_{\lfloor c \rfloor})$ and $\sigma_2 = (\phi_{\lfloor c \rfloor + 1}, \ldots, \phi_r)$
so that $\sigma = (\sigma_1, \sigma_2)$.

There are fractional ideal $\mathfrak a$ of $E$ and sublattice $N'$ of  $N$ such that $N = \mathfrak a v \oplus N'$.
Then $\sigma_2(N) = \sigma_2(N')$, which is a sublattice of rank at most $n - 1$ in $I_{r - \lfloor c \rfloor}$.
Therefore, if $g_{\mathcal O}^*(n-1)$ is finite, then $\sigma_2(N)$ is represented by some $I_s$ with $s \leq g_{\mathcal O}^*(n-1)$.
Let $\tau$ be a representation of $\sigma_2(N)$ by $I_s$.  Then $\sigma': = (\sigma_1, \tau\circ\sigma_2)$ is a representation of $N$
by $I_{s + \lfloor c \rfloor}$, which means that $g_{\mathcal O}^*(n) \leq g_{\mathcal O}^*(n-1) + \lfloor c \rfloor$.

Thus, we need to show that $g_{\mathcal O}^*(1)$ is finite.  By \cite[Theorem 2.12]{JP94} again, there exists a constant $d$ such that
a positive definite lattice $L$ of rank 1 is represented by $I_3$ provided that $\mu(L) \geq d$.  There are only finitely many isometry
classes of positive definite rank 1 lattices whose minima are smaller than $d$.  It follows that $g_{\mathcal O}^*(1)$ is finite.

\end{document}